\documentclass[12pt,reqno]{amsart}

\usepackage{amsmath,amsfonts,amssymb,amscd,amsthm,calc,enumerate}

\newtheorem{theorem}{Theorem}[section]

\newtheorem{proposition}[theorem]{Proposition}

\theoremstyle{definition}
\newtheorem{definition}[theorem]{\bf Definition}
\newcommand{\darrow}{\!\downarrow}
\newcommand{\uarrow}{\!\uparrow}

\newcommand{\la}{\langle}
\newcommand{\ra}{\rangle}

\newcommand{\bigset}[1]{\big\{ #1 \big\}}

\renewcommand{\leq}{\leqslant}
\renewcommand{\geq}{\geqslant}

\newcommand{\fa}{\forall}
\newcommand{\ex}{\exists}
\newcommand{\vph}{\varphi}

\newcommand{\FPF}{\mathrm{FPF}}

\newenvironment{reqlist}{
\begin{list}{-}
 {
\setlength{\parskip}{0cm}
\setlength{\topsep}{.3cm}
\setlength{\partopsep}{0mm}
\setlength{\rightmargin}{0cm}
\setlength{\listparindent}{0cm}
\setlength{\itemindent}{0cm}
\setlength{\parsep}{0mm}
\setlength{\leftmargin}{2.0cm}
\setlength{\labelsep}{1.0cm}
\setlength{\itemsep}{.35cm}
\settowidth{\labelwidth}{zzzzzzzz}
\setlength{\leftmargin}{\labelwidth+\labelsep}
 }
}
{\end{list}}

\begin{document}

\title[The noneffectivity of Arslanov's completeness criterion]
{The noneffectivity of Arslanov's completeness criterion 
and related theorems}

\author[S. A. Terwijn]{Sebastiaan A. Terwijn}
\address[Sebastiaan A. Terwijn]{Radboud University Nijmegen\\
Department of Mathematics\\
P.O. Box 9010, 6500 GL Nijmegen, the Netherlands.
} \email{terwijn@math.ru.nl}

\begin{abstract}
We discuss the (non)effectivity of Arslanov's completeness criterion.
In particular, we show that a parameterized version, similar to the 
recursion theorem with parameters, fails. 
We also discuss the parameterized version of another extension of the 
recursion theorem, namely Visser's ADN theorem. 
\end{abstract}


\date{\today}

\maketitle

\section{Introduction}

Kleene's recursion theorem \cite{Kleene} states that every 
computable operation on codes of partial computable functions 
has a fixed point. 
That is, for every computable function~$f$ there exists a number $e$ 
such that $\vph_{f(e)} = \vph_e$. 
Here $\vph_e$ denotes the $e$-th partial computable function. 
Kleene also proved a version of this theorem with parameters:

\begin{theorem} {\rm (The recursion theorem with parameters,
Kleene~\cite{Kleene})} \label{recthmparam}
Let $h(n,x)$ be a computable binary function.
Then there exists a computable function $f$ such that for all~$n$,
$\vph_{f(n)} = \vph_{h(n,f(n))}$.
\end{theorem}

This result shows that the recursion theorem is effective, in the 
sense that the fixed points of a computable sequence of functions
can be found in a uniformly computable way. 

The recursion theorem has been extended in several ways. 
We refer the reader to Soare~\cite{Soare} for a general discussion.
In this paper we discuss the effectivity of two extensions, 
namely Arslanov's completeness criterion 
(sections~\ref{sec:Arslanov} and~\ref{sec:Arslanovparam}) and 
Visser's ADN theorem (sections~\ref{sec:ADN} and~\ref{sec:ADNparam}).
In particular we show that the parameterized versions of these 
extensions, analogous to Theorem~\ref{recthmparam}, fail.

Our notation from computability theory is mostly standard.
Partial computable (p.c.) functions are denoted by lower case
Greek letters, and (total) computable functions by lower case
Roman letters.
$\omega$ denotes the natural numbers,
$\vph_e$ denotes the $e$-th p.c.\ function,
and $W_e$ denotes the domain of $\vph_e$.
We write $\vph_e(n)\darrow$ if this computation is defined,
and $\vph_e(n)\uarrow$ otherwise.
$\emptyset'$ denotes the halting set.
For unexplained notions we refer to Odifreddi~\cite{Odifreddi} or
Soare~\cite{Soare}.

In the discussion below we will use the following notions from 
the literature:
\begin{enumerate}[\rm (i)]

\item[$\bullet$] 
A function $f$ is called {\em fixed point free}, or simply FPF, 
if $W_{f(n)}\neq W_n$ for every~$n$.
We will also use this terminology for partial functions, 
see Definition~\ref{def} below. 

\item[$\bullet$] 
A function $g$ is called {\em diagonally noncomputable\/}, or DNC, 
if $g(e)\neq \vph_e(e)$ for every~$e$.

\end{enumerate}
Though the notions of FPF and DNC function are different, 
it is well-known that they coincide on Turing degrees, 
cf.\ Jockusch et al.~\cite{Jockuschetal}. 
Namely, a set computes a FPF function if and only if it 
computes a DNC function. 
Moreover, this is also equivalent to computing a function 
$f$ such that $\vph_{f(e)}\neq \vph_e$ for every~$e$.

DNC functions played an important role in Ku\v{c}era's 
alternative solution to Post's problem~\cite{Kucera}. 
In the paper by Kjos-Hanssen, Merkle, and Stephan~\cite{KHMS},
the notion of DNC function is linked to sets with 
high initial segment Kolmogorov complexity.

\section{Arslanov's completeness criterion} \label{sec:Arslanov}

By the recursion theorem, and the equivalence quoted above, 
no FPF function is computable.   
It is easy to see that the halting set $\emptyset'$ computes a FPF 
function, as $\emptyset'$ can list all computable functions. 
However, by the low basis theorem \cite{JockuschSoare},
there also exist FPF functions of low degree.
The next result shows that FPF functions cannot have incomplete c.e.\ 
degree. (On the other hand, by Ku\v{c}era~\cite{Kucera}, any FPF degree 
below $\emptyset'$ bounds a noncomputable c.e.\ degree.)
This shows that the recursion theorem can be extended from 
computable functions to functions bounded by an incomplete c.e.\ degree. 

\begin{theorem} {\rm (Arslanov completeness criterion \cite{Arslanov})}
\label{Arslanov}
A c.e.\ set $A$ is Turing complete if and only if 
$A$ computes a $\FPF$ function.
\end{theorem}
\begin{proof} 
Suppose $A$ is c.e.\ and incomplete, and $f\leq_T A$. 
Then $f$ has a computable approximation $\hat f(n,s)$,
and there is an $A$-computable modulus function $m(n)$
such that $\fa s\geq m(n) \big(f(n) = \hat f(n,s)\big)$.
By the recursion theorem with parameters (Theorem~\ref{recthmparam}), 
let $h$ be a computable function such that 
$$
W_{h(n)} = 
\begin{cases}
W_{\hat f(h(n),s_n)} & 
\text{if $n\in\emptyset'$ and $s_n$ is minimal such that $n\in\emptyset'_s$}, \\
\emptyset & \text{otherwise.}
\end{cases}
$$
Then there exists $n\in\emptyset'$ such that 
$\hat f(h(n),s_n) = f(h(n))$, so that $h(n)$ is a fixed 
point of~$f$. Namely, if this were not the case, then we would have 
that for all~$n$, if $n\in\emptyset'$, then $m(h(n))>s_n$, and hence 
$n\in\emptyset'_{m(h(n))}$. Thus we would have $\emptyset'\leq_T A$, 
contrary to assumption. 
\end{proof}

The proof given here is basically the contrapositive of 
the proof in Soare~\cite{Soare}.
The proof above already suggests that the result is not effective:
It does not give a fixed point effectively, but merely 
produces a c.e.\ set, namely $\bigset{h(n)\mid n\in\emptyset'}$, 
at least one of the elements of which is a fixed point. 
That this is necessarily so follows from the result in 
the next section.

\section{The failure of Arslanov's completeness criterion with parameters}
\label{sec:Arslanovparam}

Let $h$ be a computable function of two arguments. 
Since for every fixed $n$ the function $h(n,x)$ is a computable 
function of~$x$, by the recursion theorem we have 
$$
\fa n \ex x \; \vph_x = \vph_{h(n,x)}.
$$
When we Skolemize this formula we obtain:
$$
\ex f \fa n \; \vph_{f(n)} = \vph_{h(n,f(n))}.
$$ 
The recursion theorem with parameters tells us that we can 
take $f$ {\em computable\/} here. In other words, the 
recursion theorem holds uniformly.

Now consider the Arslanov completeness criterion. 
Let $A$ be an incomplete c.e.\ set, and let $h\leq_T A$
be a binary function. By Theorem~\ref{Arslanov} we have 
$$
\fa n \ex x \; \vph_x = \vph_{h(n,x)}
$$
and Skolemization gives 
$$
\ex f \fa n \; \vph_{f(n)} = \vph_{h(n,f(n))}.
$$ 
We prove that in general we cannot take $f$ computable in this case. 
This even fails when $A$ is of low Turing degree. 
Note that by relativizing the recursion theorem with parameters, 
there always exists an $A$-computable Skolem function~$f$

\begin{theorem} {\rm (Failure of Arslanov with parameters)}
There exist a low c.e.\ set $A$ and an $A$-computable 
binary function $h$ such that for every computable $f$, 
there exists $n$ with 
$$
W_{f(n)} \neq W_{h(n,f(n))}. 
$$
\end{theorem}
\begin{proof}
We build $A$ c.e.\ and $h\leq_T A$ total using a 
finite injury construction. 
The requirements for the construction are:

\begin{reqlist}

\item[$R_e:$] $f=\{e\}$ is total \; $\Longrightarrow$ \; 
$\ex n \; W_{f(n)} \neq W_{h(n,f(n))}$,

\item[$L_e:$] $\ex^\infty s \; \{e\}^{A_s}_s(e)\darrow \;\Longrightarrow\; 
\{e\}^A(e)\darrow$.

\end{reqlist}
The requirements $L_e$ guarantee that $A$ is low (cf.\ Soare~\cite{Soare}), 
and clearly the requirements $R_e$ are sufficient to prove the theorem.  
We give the requirements the following priority ordering:
$$
L_0 > R_0 > L_1 > R_1 > L_2 > \ldots
$$
To satisfy $L_e$ we do not have to enumerate anything into $A$, 
we only maintain a restraint function $r(e,s)$ to preserve computations 
in the usual way. 
Let us consider the strategy for $R_e$ in isolation. Suppose we have 
picked $n$ as a potential witness for $R_e$. 

{\em Step 1.} 
Suppose we see at stage $s$ such that $f(n) = \{e\}_s(n)\darrow$. 

If $W_{f(n),s}\neq\emptyset$ we let $W_{h(n,f(n))}=\emptyset$, 
thus satisfying $R_e$ forever. 

If $W_{f(n),s}=\emptyset$ we let $W_{h(n,f(n))}\neq\emptyset$. 

{\em Step 2.} Suppose that at a later stage $t>s$ we see 
$W_{f(n),t} = W_{h(n,f(n))}\neq\emptyset$. Now we {\em change\/} 
$h(n,f(n))$ so that $W_{h(n,f(n))}=\emptyset$ by changing $A$
below the use of~$h$. 

Since the definition of $W_{h(n,f(n))}$ needs to be adapted at most twice
(from empty to nonempty to empty), 
we can get by by letting $h$ use only two bits of $A$. We define $h$ as follows. 
We use a standard computable pairing function $\la\cdot,\cdot\ra$ to denote 
coded pairs and triples. 
For ease of notation, we write $A(x,y,z)$ instead of $A(\la x,y,z\ra)$. 
We let $h$ be an $A$-computable function such that
\begin{eqnarray*}
W_{h(n,x)}=\emptyset &\Longleftrightarrow &A(n,x,0)=A(n,x,1),\\
W_{h(n,x)} \neq \emptyset &\Longleftrightarrow &A(n,x,0)\neq A(n,x,1).
\end{eqnarray*}
Clearly such a function $h$ can be defined from $A$. (As the 
computation of $h(n,x)$ uses only two bits from~$A$, this is 
even a btt-reduction.) 

We construct $A$ in stages. 
$L_e$ requires attention at stage $s$ if $e<s$, 
$\{e\}^{A_s}_s(e)\darrow$, and $r(e,s)=0$. 
(This means that a restraint should be set to preserve the computation.)
$R_e$ requires attention at stage $s$ if $e<s$ and one of the following holds:
\begin{enumerate}[(a)]

\item $R_e$ does not have a witness at stage $s$, that is, $n_{e,s}$ is undefined. 
Required action in this case: pick $n$ larger than all current restraints 
$r(i,s)$, $i\leq e$, and also different from all other witnesses $n_{i,s}$ 
that are currently defined, and let $n_{e,s+1}=n$. 

\item $n=n_{e,s}$ is defined, $f(n)=\{e\}_s(n)\darrow$, and one of the 
following subcases applies:
\begin{enumerate}[(b.1)]

\item $W_{f(n),s}=\emptyset$ and $A_s(n,f(n),0)= A_s(n,f(n),1)=0$. 
Required action: Define $A(n,f(n),0)=1$.

\item $W_{f(n),s}\neq\emptyset$, $A_s(n,f(n),0)=1$, and $A_s(n,f(n),1)=0$. 
Required action: Define $A(n,f(n),1)=1$.
\end{enumerate} 
Also, if either $A(n,f(n),0)=1$ or $A(n,f(n),1)=1$ is set at stage~$s$, 
we define \mbox{$r(i,s+1)$} $=$ $0$ 
for all $i>e$.\footnote{That is, if $R_e$ enumerates an element into $A$, 
we drop the restraints of all lower priority requirements $L_i$. 
This is overkill since the action may not actually injure all of 
these, but it is just as easy.}
\end{enumerate} 

{\em Construction.} Initially $A$ is empty: $A_0=\emptyset$. 
At stage $s>0$, pick the highest priority requirement $R_e$ or $L_e$, 
if any, that requires attention. If there is none, proceed to the 
next stage. 
If $L_e$ is picked, set \mbox{$r(e,s+1)$} equal to the use 
of $\{e\}^{A_s}_s(e)$ (this computation converges since $L_e$ 
requires attention). Also, initialize all lower priority $R_i$ 
by letting all witnesses $n_{i,s+1}$ with $i\geq e$ be undefined, 
and proceed to the next stage. 
If $R_e$ is picked, perform the actions indicated above 
under (a) and~(b).
This concludes the construction of $A=\bigcup_s A_s$. 

{\em Verification.} 
We verify that all requirements are satisfied. 
For $L_e$, note that the only requirements that can injure it are 
the $R_i$ with $i<e$, and by induction each of these enumerates at 
most finitely many numbers into~$A$, so $L_e$ is injured at most 
finitely often, and hence is eventually satisfied. 

For $R_e$, suppose that $f=\{e\}$ is total. 
By induction, assume that no higher priority requirement $L_i$
or $R_i$ requires attention after stage~$t$. 
Let $r$ be the maximum of all higher priority restraints:
$$
r = \max_{i\leq e} \lim_{s\rightarrow\infty} r(i,s).
$$
Note that since by assumption every $L_i$, $i\leq e$, acts only 
finitely often, this limit exists and is finite. 
By the construction and (a) above, at some stage $s$ after 
the last stage that a requirement $L_i$ with $i\leq e$ acts, 
$n = n_{e,s} >r$ is defined, which is then never redefined later. 
We have the following cases. 

If $W_{f(n)}=\emptyset$, then $R_e$ acts exactly once after 
the stage~$s$ where $n$ is defined, 
the clause (b.1) applies at that stage, 
and we have $A(n,f(n),0)=1$ and $A(n,f(n),1)=0$. 
Hence $W_{h(n,f(n))}\neq\emptyset$, and $R_e$ is satisfied.

If $W_{f(n)}\neq\emptyset$ then we have two subcases:  
\begin{enumerate}[\rm (i)]

\item[$\bullet$] 
After the stage $s$ where $n$ is defined, 
$R_e$ never requires attention.
In this case we have $A(n,f(n),0)=A(n,f(n),1)=0$, 
hence $W_{h(n,f(n))}=\emptyset$, and $R_e$ is satisfied.

\item[$\bullet$] 
In the opposite case, $R_e$ does require attention after
stage~$s$. In this case, $R_e$ will act precisely twice 
after stage~$s$. 
The first time, at stage $s'$ say, since $A_s(n,f(n),0)=0$
we will have $W_{f(n),s'}=\emptyset$ 
(for otherwise $R_e$ would not require attention)
and case (b.1) will apply. 
The second time will occur at a stage $s''>s'$ that is 
large enough to see that $W_{f(n),s''}\neq\emptyset$. 
Since now $A_{s''}(n,f(n),0)=1$, case (b.2) applies, 
and we will have $A(n,f(n),0)=A(n,f(n),1)=1$.
Hence $W_{h(n,f(n))}=\emptyset$, and $R_e$ is satisfied.

\end{enumerate}
So we see that $R_e$ acts at most twice after the last 
time it is initialized, and is eventually satisfied. 
\end{proof}

\section{The ADN theorem} \label{sec:ADN}

It is well-known that Kleene found the recursion theorem by 
studying the $\lambda$-calculus. Also motivated by the $\lambda$-calculus, 
arithmetic provability, and the theory of numerations,
Visser~\cite{Visser} proved the following generalization of the
recursion theorem. It has interesting applications in the theory
of numerations, see for example Bernardi and Sorbi~\cite{BernardiSorbi}.
ADN theorem stands for ``anti diagonal normalization theorem''.

\begin{definition} 
We extend the definition of FPF function to partial functions. 
We call a partial function $\delta$ FPF if it is fixed point 
free on its domain, i.e.\ for every $n$,
\begin{equation}
\delta(n)\darrow \; \Longrightarrow \; W_{\delta(n)}\neq W_n. 
\end{equation}
\end{definition}

\begin{theorem} {\rm (ADN theorem, Visser~\cite{Visser})} \label{ADN}
Suppose that $\delta$ is a partial computable fixed point free
function. Then for every partial computable function $\psi$ there
exists a computable function $f$ such that for every $n$,
\begin{align}
\psi(n)\darrow \; &\Longrightarrow \; W_{f(n)}= W_{\psi(n)} \label{totalize} \\
\psi(n)\uarrow \; &\Longrightarrow \; \delta(f(n))\uarrow   \label{avoid}
\end{align}
\end{theorem}
If \eqref{totalize} holds for every~$n$, we say that 
$f$ {\em totalizes\/} $\psi$, 
and if in addition \eqref{avoid} holds, we say that 
$f$ {\em totalizes\/} $\psi$ {\em avoiding\/}~$\delta$.

Just as the Arslanov completeness criterion extends the 
recursion theorem from computable functions to functions 
computable from any incomplete c.e.\ degree, 
Theorem~\ref{ADN} can be extended to such degrees. 
This gives the following joint generalization of the 
ADN theorem and the Arslanov completeness criterion:

\begin{theorem} {\rm (Joint generalization \cite{Terwijn})} \label{joint}
Suppose $A$ is a c.e.\ set such that $A <_T \emptyset'$.
Suppose that $\delta$ is a partial $A$-computable $\FPF$ function. 
Then for every partial computable function $\psi$ there
exists a computable function $f$ totalizing $\psi$ avoiding $\delta$,
i.e.\ such that for every~$n$ \eqref{totalize} and \eqref{avoid} 
above hold.
\end{theorem}

Note that Theorem~\ref{joint} implies Theorem~\ref{Arslanov}, 
because if $\delta$ were total then \eqref{avoid} could not 
hold. Hence no total FPF function of incomplete c.e.\ degree 
can exist. 

Thus we have the picture of generalizations of the 
recursion theorem from Figure~\ref{fig}. 
\begin{figure}[t] 
\begin{center} 
\setlength{\unitlength}{1mm}
\begin{picture}(24, 35)
\put(12,0){\makebox[0cm][c]{recursion theorem}} 
\put(-6,15){\makebox[-.5cm][c]{ADN theorem}}
\put(-6,20){\line(2,1){14}}   
\put(-6,12.5){\line(2,-1){14}}
\put(30,15){\makebox[.5cm][c]{Arslanov}}
\put(30,20){\line(-2,1){14}}
\put(30,12.5){\line(-2,-1){14}}
\put(12,30){\makebox[0cm][c]{Theorem~\ref{joint}}}
\end{picture}
\end{center}
\caption{Generalizations of the recursion theorem\label{fig}}
\end{figure}
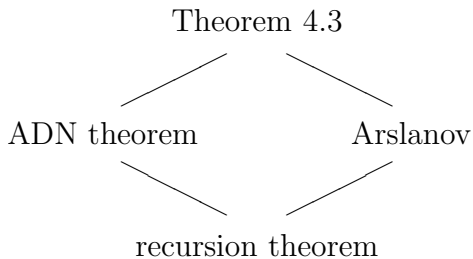
All of these generalizations can be proved using the 
recursion theorem with parameters (Theorem~\ref{recthmparam}). 
This prompts the question whether any of these generalizations
have a parameterized version. The negative answer for 
Arslanov's completeness criterion was already given in 
section~\ref{sec:Arslanovparam}. We discuss the ADN theorem in 
the next section.

\section{The ADN theorem with parameters} \label{sec:ADNparam}

The ADN theorem is uniform in codes of $\psi$, as is easy 
to see, cf.\ \cite{Terwijn}. In fact, one may assume without 
loss of generality that the function $\psi$ is universal. 
Also, from the proof of the ADN theorem from the 
recursion theorem with parameters, given in \cite{Terwijn}, 
it is clear that the code of the function $f$ depends 
effectively on a code for~$\delta$. 
Hence the result is uniform in both $\psi$ and $\delta$.
However, that the result is effective in this sense does not 
mean it has a parameterized version analogous to the 
recursion theorem with parameters. As the ADN theorem is a 
statement about partial FPF functions, and hence in a way 
a contrapositive of the recursion theorem, it is not even 
immediately clear what the statement of the ADN theorem 
with parameters should be. At least it should 
imply Theorem~\ref{recthmparam}.

To formulate the analog of Theorem~\ref{recthmparam}  
for the ADN theorem, we define the following notion.

\begin{definition} \label{def}
A partial binary function $\delta(n,x)$ is $\FPF^+$ if for every 
computable function $g$ there exists $n$ such that either 
$\delta(n,g(n))\uarrow$ or $\vph_{g(n)}\neq \vph_{\delta(n,g(n))}$.
\end{definition}

Note that by negating the property from the definition,  
$\delta$ is {\em not\/} $\FPF^+$ if there exists a computable 
function $g$ such that for every~$n$, $\delta(n,g(n))$ is defined 
and $\vph_{g(n)}=\vph_{\delta(n,g(n))}$. 
This expresses that $g$ uniformly computes fixed points for the 
family of functions $\delta(n,x)$.
By the recursion theorem with parameters, every {\em total\/} computable 
$\delta$ is not $\FPF^+$. 

We can now formulate the analog of the recursion theorem with 
parameters as follows. 

\begin{proposition} {\rm (ADN theorem with parameters)} \label{ADNparam}
Suppose that $\delta$ 
is a binary partial computable $\FPF^+$ function. 
Then for every partial computable function $\psi$ there
exists a computable function $f$ such that for every~$n$,
\begin{align}
\psi(n)\darrow \; &\Longrightarrow \; W_{f(n)}= W_{\psi(n)} \label{totalize2}\\
\psi(n)\uarrow \; &\Longrightarrow \; \delta(n,f(n))\uarrow \label{avoid2}
\end{align}
\end{proposition}

To show that this is the proper analog of Theorem~\ref{recthmparam}  
for the ADN theorem, 
we show that Proposition~\ref{ADNparam} both implies 
Theorem~\ref{recthmparam} and the ADN theorem. 
We then proceed to show that it is false (so we have used 
the term ``proposition'' here in the sense of ``mathematical 
statement'', not in the sense of ``theorem''). 

Proposition~\ref{ADNparam} implies Theorem~\ref{recthmparam}:
Note that for the proposition to hold, $\delta$ cannot be total 
(for then \eqref{avoid2} could not hold in case $\psi$ is nontotal). 
So if $\delta$ is total, it is not $\FPF^+$.
As already observed above, this means that there is a computable 
function $g$ such that for every~$n$, 
$\vph_{g(n)}=\vph_{\delta(n,g(n))}$, which is the statement 
of Theorem~\ref{recthmparam}.

Proposition~\ref{ADNparam} implies Theorem~\ref{ADN}:
Given a unary p.c.\ FPF function $\delta$, consider the
function defined as $\hat\delta(n,x) = \delta(x)$ for every 
$n$ and~$x$. 
Note that $\hat\delta$ is $\FPF^+$: For every computable 
function~$g$ and every~$n$, 
$\hat\delta(n,g(n)) = \delta(g(n))\uarrow$ or
$\vph_{g(n)}\neq \vph_{\delta(g(n))}$ since $\delta$ is FPF. 
Applying Proposition~\ref{ADNparam} to $\hat\delta$ gives, 
for a given p.c.\ $\psi$, a computable $f$ totalizing $\psi$
such that 
$$
\psi(n)\uarrow \; \Longrightarrow \; \hat\delta(n,f(n))\uarrow 
\; \Longrightarrow \; \delta(f(n))\uarrow 
$$
for every~$n$, hence Theorem~\ref{ADN} holds for~$\delta$.

\begin{theorem}
Proposition~\ref{ADNparam} is false. 
\end{theorem}
\begin{proof}
We construct $\delta$ p.c.\ and $\FPF^+$ and $\psi$ p.c.\ 
to diagonalize against all computable $f=\vph_e$,  
making sure that either \eqref{totalize2} or \eqref{avoid2} fails.

{\em Step 1}. Pick a fresh witness $n$ (i.e.\ a hitherto unused 
number) such that $\psi(n)$ is still undefined. 
To be total, $f$ has to converge on~$n$. If this never happens, 
we do not have to take any further action.   

{\em Step 2}. If $f(n)$ as in Step~1 becomes defined, let 
$\delta(n,f(n))\darrow$. This kills $f$ by making \eqref{avoid2} fail, 
but we still have to ensure that $\delta$ is $\FPF^+$. 
To meet the condition from Definition~\ref{def}, we need a 
number $m$ such that 
$\delta(m,f(m))\uarrow$ or $\vph_{f(m)}\neq \vph_{\delta(m,f(m))}$.
We can simply do this by picking $m\neq n$, waiting for 
$f(m)\darrow$, and letting $\delta(m,f(m))\uarrow$.
The only problem with this scenario is that $f$ could be a 
constant function. But we can use the property \eqref{totalize2} 
to {\em force\/} $f$ to be nonconstant, using that we are free to 
define~$\psi$. For example, if we take $\psi$ universal we 
know $f$ will have to be nonconstant to satisfy \eqref{totalize2}.
So then we can simply wait for $f$ to become defined on a number
$m\neq n$, and let $\delta(m,f(m))\uarrow$. If $f$ fails to do 
this, it fails to satisfy~\eqref{totalize2}.

Obviously, we can carry out steps~1 and 2 for every $f=\vph_e$
by simply picking different witnesses for each of them. 
This proves that for the $\delta$ and $\psi$ constructed, there 
does not exist $f$ as in Proposition~\ref{ADNparam}.
\end{proof}

\end{document}